\documentclass[a4paper,12pt]{article}
\usepackage[utf8]{inputenc}
\usepackage{amsmath}
\usepackage{amssymb}
\usepackage{amsthm}
\usepackage{cite}
\usepackage{diagbox}
\usepackage{tikz}
\usetikzlibrary{arrows}
\usetikzlibrary{positioning}

\title{Numerical analysis of a singularly perturbed convection diffusion problem with shift in space}
\author{Mirjana Brdar,\footnote{Faculty of Technology Novi Sad, University of Novi Sad, Serbia,\newline \mbox{e-mail}: mirjana.brdar@uns.ac.rs}
        \,
        Sebastian Franz,\footnote{corresponding author, Institute of Scientific Computing, Technische Universit\"at Dresden, Germany,\newline  \mbox{e-mail}: sebastian.franz@tu-dresden.de}
        \,
        Lars Ludwig,\footnote{Institute of Scientific Computing, Technische Universit\"at Dresden, Germany,\newline  \mbox{e-mail}: lars.ludwig@tu-dresden.de}
        \,
        Hans-G\"{o}rg Roos,\footnote{Institute of Numerical Mathematics, Technische Universit\"at Dresden, Germany,\newline  \mbox{e-mail}: hans-goerg.roos@tu-dresden.de}}
\date{\today}

\makeatletter
\renewcommand*\env@matrix[1][r]{\hskip -\arraycolsep
  \let\@ifnextchar\new@ifnextchar
  \array{*\c@MaxMatrixCols #1}}
\makeatother

\newcommand{\e}{\mathrm{e}}
\newcommand{\eps}{\varepsilon}
\newcommand{\norm}[2]{\|{#1}\|_{#2}}
\newcommand{\tnorm}[1]{\left|\!\!\;\left|\!\!\;\left| {#1}
                       \right|\!\!\;\right|\!\!\;\right|}

\newcommand{\R}{\mathbb{R}}

\newcommand{\U}{\mathcal{U}}
\newcommand{\PS}{\mathcal{P}}
\newcommand{\scp}[1]{\left\langle #1 \right\rangle}
\newcommand{\rarrow}{\quad\Rightarrow\quad}
\newcommand{\qmbox}[1]{\quad\mbox{#1}\quad}
\newcommand{\ds}{\mathrm{d}s}
\newcommand{\dt}{\mathrm{d}t}
\newcommand{\dx}{\mathrm{d}x}
\newcommand{\jump}[1]{[\hspace*{-2pt}[#1]\hspace*{-2pt}]}

\newcommand{\dy}{\mathrm{d}y}
\newcommand{\ord}[1]{\mathcal{O}( #1 )}

\newcommand{\pmtrx}[1]{\begin{pmatrix}[c]#1\end{pmatrix}}
\theoremstyle{plain}
\newtheorem{theorem}{Theorem}[section]
\newtheorem{lemma}[theorem]{Lemma}

\newtheorem{remark}[theorem]{Remark}

\begin{document}
  \maketitle
  \begin{abstract}
    We consider a singularly perturbed convection-diffusion problem that has in addition a shift term.
    We show a solution decomposition using asymptotic expansions and a stability result.
   {Based upon this we provide a numerical analysis of high order finite element method on layer adapted meshes.}
    We also apply a new idea of using a coarser mesh in places where weak layers appear. 
    Numerical experiments confirm our theoretical results.
  \end{abstract}

  \section{Introduction}
  In this paper we want to look at the static singularly perturbed problem given by
  \begin{subequations}\label{eq:Lu}
    \begin{align}
      -\eps u''(x)-b(x) u'(x)+c(x) u(x)+d(x) u(x-1)&=f(x),\quad x\in\Omega:=(0,2),\\
      u(2) &= 0,\\
      u(x) &= \Phi(x),\quad x\in(-1,0],
    \end{align}
  \end{subequations}
  where $0<\eps\ll 1$, $b\geq \beta>0$, $d\geq 0$, $\displaystyle c-\frac{b'}{2}-\frac{\norm{d}{L_\infty(1,2)}}{2}\geq\gamma>0$.
  For the function $\Phi$ we assume $\Phi(0)=0$, which is not a restriction as a simple transformation
  can always ensure this condition. Then, it holds $u\in \U:=H^1_0(\Omega)$.
  
  The literature on singularly perturbed problems is vast, see e.g. the book \cite{RST08} and the references therein.
  But for problems that in addition also have a shift-operator, sometimes also called a delay-operator, with a large shift,
  there are not many. For the time-dependent case and a reaction-diffusion type problem there are e.g. \cite{KumarKadalbajoo, Gupta, Bansal, Chakravarthy, KumarKumari}.
  Recently, we also investigated the time dependent version of a singularly perturbed reaction-diffusion problem in \cite{BFrLR22}
  using a finite element method in time and space.
  
  For convection dominated singularly perturbed problems with an additional shift there are also publications in literature, 
  see e.g. \cite{SR12,SR13,RS15}. They all consider a negative coefficient, here called $d$, which supports a maximum principle.
  Then finite differences on layer adapted meshes of rather low order are used. In our paper we consider finite element methods
  of arbitrary order for positive coefficients $d$.
  Standard convection-diffusion problems with a fixed convection coefficient show one boundary layer near
  the outflow boundary in contrast to two boundary layers for the reaction diffusion problem. Therefore, we expect 
  the behaviour of the problem with a shift also to have some different structure that those of reaction-diffusion type.
  
  In e.g. \cite{SR12} an asymptotic expansion of the solution is given, where the direction of shift and convection is opposite, 
  but only to the lowest order. For the purpose of this paper we want a complete solution decomposition. 
  Therefore, in Section~\ref{sec:SolDec} we provide a solution decomposition 
  of $u$ into various layers and a smooth part using a different approach. We prove it rigorously for the constant coefficient case.
  In Section~\ref{sec:NumAna} a numerical analysis is provided for the discretisation using finite elements of 
  arbitrary order on a classical S-type mesh and a new coarser type of mesh. Finally Section~\ref{sec:NumEx}
  provides some numerical results supporting our analysis. We finish this paper with a technical abstract on some
  terms involving Green's function.
  
  \textbf{Notation:} 
    For a set $D$, we use the notation $\norm{\cdot}{L^p(D)}$ for the $L^p-$norm over $D$, where $p\geq 1$.
    The standard scalar product in $L_2(D)$ is marked with $\scp{\cdot,\cdot}_D$.
    If $D=\Omega$ we sometimes drop the $\Omega$ from the notation.
    Throughout the paper, we will write $A\lesssim B$ if there exists a generic positive constant $C$ independent of the perturbation parameter $\eps$ and the mesh,
    such that $A\leq C B$. We will also write $A\sim B$ if $A\lesssim B$ and $B\lesssim A$.
  
  \section{Solution decomposition}\label{sec:SolDec}
  The considered problem with a shift term has some different properties compared to a convection-diffusion problem without
  the shift. One of the major ones is, that it is unknown whether a maximum principle holds for $d\geq 0$. In the case $d\leq 0$ a maximum principle
  is proved in e.g. \cite{SR12}, but the proof cannot be applied here.
  For the following solution decomposition we will need a stability result that is provided in the next theorem for the case of constant coefficients
  and assumed to hold true in the general case of variable coefficient.
  \begin{theorem}\label{thm:stability}
    Consider the problem: Find $u=u_1\chi_{(0,1)}+u_2\chi_{(1,2)}$,  where $\chi_D$ is the characteristic function of $D$ such that it holds
    \begin{align*}
      -\eps u_1''(x)-bu_1'(x)+cu_1(x)&=f(x)
      ,\,x\in(0,1),\\
      u_1(0)=0,\,u_1(1)&=\alpha,\\
      -\eps u_2''(x)-bu_2'(x)+cu_2(x)&=f(x)-du_1(x-1),\,x\in(1,2),\\
      u_2(2)=\beta,\,u_2(1)&=\alpha + \delta
    \end{align*}
    with constant $b>0$ and $c>0$, arbitrary boundary value $\beta\in\R$ and jump $\delta\in\R$, 
    and $\alpha\in\R$ chosen, such that $u_1'(1^-)=u_2'(1^+)$.
    Then we have
    \[
      \norm{u}{L^\infty}\lesssim\norm{f}{L^\infty(0,2)}+|\beta|+|\delta|.
    \]
  \end{theorem}
  \begin{proof}
    Each of these two sub-problems is a standard convection diffusion problem 
    with a known Green's function $G$ for the case of homogeneous boundary conditions on $(0,1)$
    that can be constructed as shown e.g. in \cite[Chapter 1.1]{Melnikov2012}.
    Thus, we can, using $\hat b(t):=b-ct$ and $\hat c(t):=c+dt$, represent the solutions as
    \begin{align*}
      u_1(x) & = \alpha x+\int_0^1 G(x,t)\left(f(t)+\alpha \hat b(t) \right)\dt,\\
      u_2(x) & = (\alpha+\delta)(2-x)+\beta(x-1)\\&\hspace*{1.5cm}
                  +\int_0^1 G(x-1,t)\bigg( f(t+1)-(\alpha+\delta)(\hat b(t)+\hat c(t))+\beta\hat b(t)\\&\hspace*{5cm}
                  -d\int_0^1G(t,s)\left(f(s)+\alpha \hat b(s) \right)\ds \bigg)\dt,
    \end{align*}
    where in the second case we have used, that $G(x-1,t-1)$ is the Green's function for the problem on $(1,2)$.
    The condition for $\alpha$ can now be written as
    \begin{align}
      \alpha
      &=\frac{N}{D}\label{eq:alpha}
      \intertext{where}
      N&:=\beta-\delta+\int_0^1 G_x(0,t)\left(f(t+1)-\delta\hat c(t)+(\beta-\delta)\hat b(t)-d\int_0^1G(t,s)f(s)\ds\right)\dt\notag\\
        &\hspace{2cm}-\int_0^1 G_x(1 ,t)f(t)\dt\notag
      \intertext{and}
      D&:=2+\int_0^1 G_x(1,t)\hat b(t)\dt+\int_0^1 G_x(0,t)\left(\hat b(t)+\hat c(t)+d\int_0^1G(t,s)\hat b(s)\ds\right)\dt.\notag
    \end{align}
    For the given problem we can compute all relevant information concerning the Green's function, see Appendix \ref{app:a}, and can estimate
    \[
      |N| \lesssim |\beta|+|\delta|+\left(|\beta|+|\delta|+\norm{f}{L^\infty(0,2)}\right)\cdot\frac{1}{\eps}\quad\text{and}\quad
      D \gtrsim \frac{1}{\eps}.
    \]
    Therefore,
%
%
    we have
    \[
      |\alpha|\lesssim |\beta|+|\delta|+\norm{f}{L^\infty(0,2)}.
    \]
    Now using the representations of $u_1$ and $u_2$ and 
    \[
      \norm{G(x,\cdot)}{L^1(0,1)}\lesssim 1,\quad\text{for all }x\in(0,1)
    \]
    we have proved the assertion.
  \end{proof}

  Contrary to the reaction-diffusion case, where in addition to the boundary layers a strong inner layer forms, see \cite{BFrLR22},
  the convection-diffusion case has only a strong boundary layer at the outflow boundary and a weak inner layer. 
  
  \begin{theorem}\label{thm:decom}
    Let $k\geq 0$ be a given integer and the data of \eqref{eq:Lu} smooth enough. Then it holds
    \[
      u=S+E+W,
    \]
    where 
    for any $\ell\in\{0,1,\dots,k\}$ it holds
    \begin{align*}
      \norm{S^{(\ell)}}{L_2(0,1)}+\norm{S^{(\ell)}}{L_2(1,2)} & \lesssim 1,&
      |E^{(\ell)}(x)|&\lesssim \eps^{-\ell}\e^{-\beta \frac{x}{\eps}},\quad x\in[0,2],\\
      &&|W^{(\ell)}(x)|&\lesssim \begin{cases}
                                   0,&x\in(0,1),\\
                                   \eps^{1-\ell}\e^{-\beta \frac{(x-1)}{\eps}},& x\in(1,2).
                                 \end{cases}
    \end{align*}
  \end{theorem}
  \begin{proof}
    We prove this theorem using asymptotic expansions. For simplicity we assume $b,c$ and $d$ to be constant. 
    Adjusting the proof for variable smooth coefficients is straightforward using Taylor expansions and assuming 
    Theorem~\ref{thm:stability} to hold true for variable coefficients.
    
    We start by writing the problem using $u_1$ and $u_2$ as the solution on $(0,1)$ and $(1,2)$ resp.
    \begin{align*}
      -\eps u_1''(x)-bu_1'(x)+cu_1(x) &= f(x)-d\Phi(x -1),\quad x\in(0,1),\\
      -\eps u_2''(x)-bu'_2(x)+cu_2(x) &= f(x)-du_1(x -1),\quad x\in(1,2),\\
      u_1(0)  = 0,\quad
      u_1(1) &= u_2(1),\quad
      u_1'(1) = u_2'(1),\quad
      u_2(2)  = 0.
    \end{align*}
    Let
    $\sum_{i=0}^k\eps^i (S_{i,-}\chi_{[0,1)}+S_{i,+}\chi_{[1,2]})$ be the outer expansion and by
    substituting this into the differential system we obtain
    \begin{align*}
      \sum_{i=0}^k\eps^i\left( -\eps S_{i,-}''(x)-bS_{i,-}'(x)+cS_{i,-}(x) \right)&=f(x)-d\Phi(x-1),\,x\in(0,1),\\
      \sum_{i=0}^k\eps^i\left( -\eps S_{i,+}''(x)-bS_{i,+}'(x)+cS_{i,+}(x) \right)&=f(x)-d\sum_{i=0}^k\eps^iS_{i,-}(x-1),\,x\in(1,2)      
    \end{align*}
    plus boundary conditions and continuity conditions. For the coefficient of $\eps^0$ (including some of the additional conditions) we obtain 
    \begin{align*}
      -b S_{0,-}'(x)+cS_{0,-}(x)&=f(x)-d\Phi(x-1),\,x\in(0,1),     & S_{0,-}(1) &= S_{0,+}(1),\\
      -b S_{0,+}'(x)+cS_{0,+}(x)&=f(x)-dS_{0,-}(x-1),\,x\in(1,2),  & S_{0,+}(2) &= 0.
    \end{align*}
    According to Lemma~\ref{lem:exist}, after mapping the second line to $(0,1)$, there exists a solution $S_0=S_{0,-}\chi_{[0,1)}+S_{0,+}\chi_{[1,2]}$, that is continuous
    and $S_0(2)=0$, but $S_0(0)\neq 0$. Thus, we correct this with a boundary correction using the stretched variable $\xi=\frac{x}{\eps}$
    and $\sum_{i=0}^k\eps^i\tilde E_i(\xi)$. Substituting this into the differential equation yields
    \[
      \sum_{i=0}^k\eps^i\left(-\eps^{-1}(\tilde E_i''(\xi) +b\tilde E_i'(\xi)) +c\tilde E_i(\xi))+d\tilde E_i\left(\xi-\frac{1}{\eps}\right)\chi_{(\frac{1}{\eps},\frac{2}{\eps})}\right)=0.
    \]
    We deal with the shift term later and obtain for the coefficient of $\eps^{-1}$ the boundary correction problem
    \begin{align*}
      \tilde E_0''(\xi) +b\tilde E_0'(\xi) &=0,\quad \tilde E_0(0)=-S_{0,-}(0),\,\lim_{\xi\to\infty}\tilde E_0(\xi)=0
      \quad\Rightarrow\quad
      \tilde E_0(\xi)=-S_{0,-}(0)\e^{-b\xi}.
    \end{align*}
    Furthermore, we correct the jump of the derivative of $S_0$ at $x=1$ with an inner expansion and the variable $\eta=\frac{x-1}{\eps}$.
    Using $\sum_{i=1}^{k+1}\eps^i \tilde W_i(\eta)$ we have
    \[
      \sum_{i=1}^{k+1}\eps^i\left(-\eps^{-1}(\tilde W_i''(\eta) +b\tilde W_i'(\eta)) +c\tilde W_i(\eta)\right)=0.
    \]
    For the coefficient of $\eps^{0}$ and initial conditions at $\eta=0$ it follows
    \[
      \tilde W_1''(\eta) +b\tilde W_1'(\eta) =d\tilde E_0(\eta),\quad \tilde W_1'(0)=-\jump{S_0'(1)},\,\lim_{\eta\to\infty}\tilde W_1(\eta)=0.
    \]
    Here we included the shift of $\tilde E_0$ into the differential equation. We obtain a solution
    \[
      \tilde W_1(\eta)=\tilde Q_1(\eta)\e^{-b\eta}
    \]
    where $\tilde Q_1$ a polynomial of degree 1.    
    Thus far we have
    \[
      u_0=S_0+E_0+\eps W_1\chi_{[1,2]},\,u_0(0)=0,\,|u_0(2)|\lesssim \e^{-\frac{b}{\eps}}
    \]
    and in addition
    \[
      \jump{u_0'(1)}=0,\quad \jump{u_0(1)}=W_1(1)\lesssim \eps.
    \]
    Thus we have corrected the jump in the derivative, but introduced a jump in the function value of order $\eps$.
    In order to correct this jump we continue with the same steps, now for the coefficients of $\eps^i$ for $i>0$.
    We obtain the problems
    \begin{align*}
      -b S_{i,-}'(x)+cS_{i,-}(x)&=S_{i-1,-}''(x),\,x\in(0,1),& S_{i,-}(1) &= S_{i,+}(1)-W_i(1)\\
      -b S_{i,+}'(x)+cS_{i,+}(x)&=S_{i-1,+}''(x)-dS_{i,-}(x-1),\,x\in(1,2),& S_{i,+}(2) &= 0\\
      \stackrel{Lemma~\ref{lem:exist}}{\Longrightarrow}\quad S_i&=S_{i,-}\chi_{[0,1)}+S_{i,+}\chi_{[1,2]},
    \end{align*}
    and
    \begin{align*}
      \tilde E_{i}''(\xi) +b\tilde E_{i}'(\xi) &=c\tilde E_{i-1}(\xi),& \tilde E_{i}(0)&=-S_{i,-}(0),\,\lim_{\xi\to\infty}\tilde E_i(\xi)=0\\
      \Longrightarrow\quad \tilde E_i(\xi) &=\tilde P_i(\xi)\e^{-b\xi},\\
      \tilde W_{i+1}''(\eta) + b\tilde W_{i+1}'(\eta) &= c\tilde W_{i}(\eta)+d\tilde E_{i}(\eta), &\tilde W_{i+1}'(0)&=-\jump{S_i'(1)},\,\lim_{\eta\to\infty}\tilde W_{i+1}(\eta)=0\\
      \Longrightarrow\quad \tilde W_{i+1}(\eta) &=\tilde Q_{i+1}(\eta)\e^{-b\eta},
    \end{align*}
    where $\tilde P_i$ and $\tilde Q_{i+1}$ are polynomials of degree $i$ and $i+1$, resp. 
    The following Figure~\ref{fig:dep}
    \begin{figure}[htb]
      \begin{center}
        \begin{tikzpicture}[->,>=stealth',shorten >=1pt,node distance=0.7cm, semithick]
            \node (V0)                                   {$S_0$};
            \node (Z1) [below right =of V0]              {$\tilde W_1$};
            \node (W0) [below left  =of Z1]              {$\tilde E_0$};
            \node[node distance=1cm]
                  (V1) [right       =of V0]              {$S_1$};
            \node (Z2) [below right =of V1]              {$\tilde W_2$};
            \node (W1) [below left  =of Z2]              {$\tilde E_1$};
            \node[node distance=1cm]
                  (V2) [right       =of V1]              {$S_2$};
            \node (Z3) [below right =of V2]              {$\tilde W_3$};
            \node (W2) [below left  =of Z3]              {$\tilde E_2$};
            \node (V3) [right       =of V2]              {$\dots$};
            \node (Z4) [right       =of Z3]              {$\dots$};
            \node (W3) [right       =of W2]              {$\dots$};
            \path[dotted,blue] (V0) edge (Z1);
            \path[dotted,red]  (V0) edge (W0);
            \path              (W0) edge (Z1);
            \path              (V0) edge (V1);
            \path              (W0) edge (W1);
            \path[dotted,red]  (Z1) edge (V1);
            \path              (Z1) edge (Z2);
            \path[dotted,blue] (V1) edge (Z2);
            \path[dotted,red]  (V1) edge (W1);
            \path              (W1) edge (Z2);
            \path              (V1) edge (V2);
            \path              (W1) edge (W2);
            \path[dotted,red]  (Z2) edge (V2);
            \path              (Z2) edge (Z3);
            \path[dotted,blue] (V2) edge (Z3);
            \path[dotted,red]  (V2) edge (W2);
            \path              (W2) edge (Z3);
        \end{tikzpicture}
      \end{center}
      \caption{Dependence graph of the problems in the solution decomposition.\label{fig:dep}}
    \end{figure}
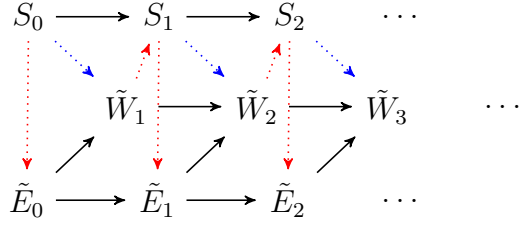
    shows in a diagram the dependence of the problems. 
    Dotted lines represent influence on boundary values, while solid ones are via the differential equation.
    
    Thus, for the expansion
    \[
      u_k:=\underbrace{\sum_{i=0}^k\eps^i S_i(x)}_{=:\tilde S(x)}
           +\underbrace{\sum_{i=0}^k\eps^iP_i\left( \frac{x}{\eps} \right)\e^{-\frac{bx}{\eps}}}_{=:E(x)}
           +\underbrace{\sum_{i=1}^{k+1}\eps^iQ_i\left( \frac{x-1}{\eps} \right)\e^{-\frac{b(x-1)}{\eps}}\chi_{[1,2]}(x)}_{=:W(x)}
    \]
    we have
    \[
      \jump{u_k(1)}=:\delta,\quad
      \jump{u_k'(1)}=0,\quad 
      u_k(0)=0,\quad
      u_k(2)=:\beta,
    \]
    where
    \[
      |\delta|\lesssim \eps^k\qmbox{and}
      |\beta|\lesssim \e^{-\frac{b}{\eps}},
    \]
    and for the remainder $R:=u_k-u$ follows the same. Finally, it holds
    \begin{align*}
      -\eps R''-bR'+cR
        &= \eps^k(S''_{k,-}+ce_k),\,\text{in }(0,1)\\
      -\eps R''-bR'+cR
        &= \eps^k(S''_{k,+}+ce_k+cw_{k+1})-dR(\cdot-1),\,\text{in }(1,2).
    \end{align*}
    Using the stability result of Theorem~\ref{thm:stability} we obtain
    \[
      \norm{R}{L^\infty}\lesssim \eps^k
    \]
    and we can set 
    \[
      S:=\tilde S+R.\qedhere
    \]
  \end{proof}

  \begin{lemma}\label{lem:exist}
    The ordinary differential system
    \begin{align*}
      -V'(x)+c_1(x)V(x)\hspace{2.3cm}&=g_1(x),\,x\in(0,1),\quad V(1)=W(0)+\alpha,\\
      -W'(x)+c_2(x)W(x)+d(x)V(x)&=g_2(x),\,x\in(0,1),\quad W(1)=0
    \end{align*}
    has for positive $d$ and any $c_1,c_2,g_1,g_2,\alpha$
    a unique solution.
  \end{lemma}
  \begin{proof}
    For $x\in(0,1)$ the system can be written as
    \[
      \pmtrx{V\\W}'(x) = \pmtrx{c_1(x)&0\\d(x)&c_2(x)}\pmtrx{V\\W}(x)-\pmtrx{g_1(x)\\g_2(x)},
      \quad
      \pmtrx{V\\W}(1)=\pmtrx{W(0)+\alpha\\0}
    \]
    or short
    \[
      \pmtrx{V\\W}' = A\pmtrx{V\\W} - g,
      \quad
      \pmtrx{V\\W}(1)=\pmtrx{W(0)+\alpha\\0}.
    \]
    With $B(x)=\int_0^x A(y)\dy$ and the matrix exponential, the solution can be represented as
    \[
      \pmtrx{V\\W}(x)=\e^{B(x)}\bigg( \e^{-B(1)}\pmtrx{W(0)+\alpha\\0}
      +\underbrace{\int_x^1\e^{-B(y)}g(y)\dy}_{=:T(x)} \bigg).
    \]
    Now, this solution is still recursively defined. In order to investigate this further, let
    \[
      B(1)=\pmtrx{C_1&0\\D&C_2},\text{ where }
      D:=\int_0^1d(x)\dx,\,
      C_i=\int_0^1 c_i(d)\dx,\,i\in\{1,2\},
    \]
    $\tilde b_{21}$ be the 2,1-component of $\e^{-B(1)}$ and $T_2$ the second component of $T$.
    Then we have
    \[
      W(0) = \tilde b_{21}(W(0)+\alpha)
      +T_2(0)
      \rarrow
      W(0) = \frac{T_2(0)+\tilde b_{21}\alpha}{1-\tilde b_{21}},
    \]
    if $\tilde b_{21}\neq 1$. Due to the assumption $d>0$ we have $D>0$ and therefore
    \[
      \tilde b_{21}:=\begin{cases}
                        -D \e^{-C_1},&C_1=C_2,\\
                         D\frac{\e^{-C_1}-\e^{-C_2}}{C_1-C_2},&C_1\neq C_2,
                      \end{cases}
    \]
    is always negative and thus not 1. 
  \end{proof}
  \begin{remark}
    The condition $d>0$ is sufficient, but not necessary. But some condition is needed,
    as can be seen by the example $c_1=c_2=1,\,d=-\e,\,\alpha=0$ 
    and for example $g_1(x)=g_2(x)=1$ 
    for which no solution $(V,W)$ exists:
    \[
      V(x)=1 + (W(0)-1)\e^{x-1},\,
      W(x)=\e^x\cdot\left((1-x)W(0)+x-2-\e^{-1}\right)  + \e + 1
    \]
    fulfils the system and the conditions $W(1)=0$, $V(1)=W(0)$, but
    \[
      W(0)=W(0)-1+\e-\e^{-1}
    \]
    is not defined. 
  \end{remark}
  \begin{remark}
    The related problem
    \begin{align*}
      -\eps u''(x)-b(x) u'(x)+c(x) u(x)+d(x) u(x+1)&=f(x),\quad x\in\Omega:=(0,2),\\
      u(0) &= 0,\\
      u(x) &= \Phi(x),\quad x\in[2,3),
    \end{align*}
    where the directions of shift and convection are opposing,
    can be analysed quite similarly, yielding the same solution decomposition as Theorem~\ref{thm:decom}.
    Here the reduced problems are always solvable, independent of $d$, but the problems 
    for the boundary correction have to be split into the two subregions.
  \end{remark}

  \section{Numerical analysis}\label{sec:NumAna}
  \subsection{Preliminaries}
  Using standard $L_2$-products and 
  integration by parts we define our bilinear and linear form by
  \begin{align}
    B(u,v)&:=\eps\scp{u',v'}_{\Omega}+\scp{cu-bu',v}_{\Omega}+\scp{du(\cdot -1),v}_{(1,2)}\notag\\
          &= \scp{f,v}_{\Omega}-\scp{d\phi(\cdot -1),v}_{(0,1)}=:F(v)
  \end{align}
  for $u,v\in \U$. With
  \begin{align*}
    -\scp{bu',u}_{\Omega}&=\scp{b'u,u}_{\Omega}+\scp{bu',u}_{\Omega}
    \intertext{and}
    \scp{du(\cdot-1),u}_{(1,2)}&\leq \frac{\norm{d}{L_\infty(1,2)}}{2}\left( \norm{u}{L_2(0,1)}^2+\norm{u}{L_2(1,2)}^2 \right)
    =\frac{\norm{d}{L_\infty(1,2)}}{2}\norm{u}{L_2}^2
  \end{align*}
  we have coercivity  w.r.t. the energy norm $\tnorm{\cdot}$
  \begin{align*}
    B(u,u) &= \eps\norm{u'}{L_2}^2+\scp{cu-bu',u}_{\Omega}+\scp{du(\cdot-1),u}_{(1,2)}\\
            &\geq \eps\norm{u'}{L_2}^2+ \scp{\left( c-\frac{b'}{2} \right)u,u}_{\Omega}-\frac{\norm{d}{L_\infty(1,2)}}{2}\norm{u}{L_2}^2\\
            &\geq \eps\norm{u'}{L_2}^2 + \gamma \norm{u}{L_2}^2=:\tnorm{u}^2
  \end{align*}
  due to our assumptions on the data.
  
  \subsection{On standard S-type meshes}
  For the construction of an S-type mesh, see \cite{RL99}, let us assume the number of cells $N$ to be divisible by 4.
  Next we define a mesh transition value
  \[
    \lambda = \frac{\sigma\eps}{\beta}\ln(N),
  \]
  with a constant $\sigma$ to be specified later. In order to have an actual layer we assume $\eps$ to be small enough. To be more precise, we assume
  \[
    \frac{\sigma\eps}{\beta}\ln(N)\leq \frac{1}{2}
  \]
  such that $\lambda\leq 1/2$ follows.
  
  Then using a monotonically increasing mesh defining function $\phi$ with $\phi(0)=0$ and $\phi(1/2)=\ln(N)$,
  see \cite{RL99} for the precise conditions on $\phi$,
  we construct the mesh nodes
  \[
    x_i=\begin{cases}
           \displaystyle\frac{\sigma\eps}{\beta}\phi\left(\frac{2i}{N}\right),&           \displaystyle0\leq i \leq \frac{N}{4},\\[1ex]
           \displaystyle\frac{4i}{N}(1-\lambda)+2\lambda-1,                    & \displaystyle\frac{N}{4}\leq i \leq \frac{N}{2},\\[1ex]
          \displaystyle 1+x_{i-N/2},                                           & \displaystyle\frac{N}{2}\leq i \leq N.
        \end{cases}
  \]
  Let us denote the smallest mesh-width inside the layers by $h$, for which holds $h\leq \eps$.
  Associated with $\phi$ is the mesh characterising function $\psi=\e^\phi$, that classifies the convergence quality
  of the meshes by the quantity $\max|\psi'|:=\max\limits_{t\in[0,1/2]}|\psi'(t)|$. Two of the most common S-type meshes are
  the Shishkin mesh with
  \[
    \phi(t) = 2t\ln N,\quad
    \psi(t) = N^{-2t},\quad
    \max|\psi'| =2\ln N
  \]
  and the Bakhvalov-S-mesh
  \[
    \phi(t)=-\ln(1-2t(1-N^{-1})),\quad
    \psi(t)=1-2t(1-N^{-1}),\quad
    \max|\psi'|= 2.
  \]
  By definition it holds
  \[
    |E(\lambda)|\lesssim  N^{-\sigma}
    \quad\text{and}\quad
    |W(1+\lambda)|\lesssim  \eps N^{-\sigma}.
  \]
  As discrete space we use 
  \[
    \U_N:=\{v\in H_0^1(\Omega):v|_\tau\in\PS_k(\tau)\},
  \]
  where $\PS_k(\tau)$ is the space of polynomials of degree $k$ at most on a cell $\tau$ of the mesh.
  Let $I$ be the standard Lagrange-interpolation operator into $\U_N$ using equidistant points or any other suitable 
  distribution of points. The derivation of the interpolation error can be done like for a standard convection-diffusion 
  problem, see e.g. \cite{RST08}. We therefore skip the proof.
  \begin{lemma}[Interpolation error estimates]\label{lem:inter}
    For $\sigma\geq k+1$, $u=S+E+W$ assuming the solution decomposition and the Lagrange interpolation operator $I$ it holds
    \begin{align*}
      \norm{u-Iu}{L^2(\Omega)}&\lesssim  (h+N^{-1}\max|\psi'|)^{k+1},\\
      \norm{(u-Iu)'}{L^2(\Omega)}&\lesssim  \eps^{-1/2}(h+N^{-1}\max|\psi'|)^k
    \end{align*}
    and additionally
    \begin{align*}
      \norm{E-IE}{L^2((0,\lambda)\cup(1,1+\lambda))}&\lesssim \eps^{1/2}(N^{-1}\max|\psi'|)^{k+1},\\
      \norm{E-IE}{L^2((\lambda,1)\cup(1+\lambda,2))}&\lesssim N^{-(k+1)},\\
      \norm{(W-IW)'}{L^2(\Omega)}&\lesssim \eps^{1/2}(N^{-1}\max|\psi'|)^k.
    \end{align*}
  \end{lemma}
  The numerical method is now given by: Find $u_N\in \U_N$ such that for all $v\in \U_N$ it holds
  \begin{equation}\label{eq:method}
    B(u_N,v)=F(v).
  \end{equation}
  Obviously, we have immediately Galerkin orthogonality
  \[
    B(u-u_N,v)=0\quad\text{for all }v\in \U_N.
  \]
  Now the convergence of our method is easily shown.
  \begin{theorem}\label{thm:conv}
    For the solution $u$ of \eqref{eq:Lu} and the numerical solution $u_N$ of \eqref{eq:method} holds on an S-type mesh
    with $\sigma\geq k+1$
    \[
      \tnorm{u-u_N}\lesssim  (h+N^{-1}\max|\psi'|)^k.
    \]
  \end{theorem}
  \begin{proof}
    We start with a triangle inequality
    \[
      \tnorm{u-u_N}\leq \tnorm{u-Iu}+\tnorm{Iu-u_N}
    \]
    where the first term can be estimated by Lemma~\ref{lem:inter}. Let $\chi:=Iu-u_N\in\U_N$ and $\psi:=u-Iu$. Then
    coercivity and Galerkin orthogonality yield
    \begin{align*}
      \tnorm{\chi}^2  
        &\leq B(\eta,\chi)
          =\eps\scp{\eta',\chi'}_{\Omega}+\scp{c\eta-b\eta',\chi}_{\Omega}+\scp{d\eta(\cdot -1),\chi}_{(1,2)},\\
      &\lesssim  (h+N^{-1}\max|\psi'|)^k\tnorm{\chi}+\scp{b(E-IE),\chi'}_{\Omega},
    \end{align*}
    where Cauchy-Schwarz inequalities and the interpolation error estimates were used for all but the convection term including the strong layer,
    where integration by parts was applied.
    For the remaining term we decompose the resulting scalar product into fine and coarse regions.
    \begin{align*}
      |\scp{b(E-IE),\chi'}_\Omega|
        &\leq |\scp{b(E-IE),\chi'}_{(0,\lambda)\cup(1,1+\lambda)}|
              +|\scp{b(E-IE),\chi'}_{(\lambda,1)\cup(1+\lambda,2)}|\\
        &\lesssim \eps^{1/2}(N^{-1}\max|\psi'|)^k\norm{\chi'}{L^2((0,\lambda)\cup(1,1+\lambda))}
              +N^{-(k+1)}\norm{\chi'}{L^2((\lambda,1)\cup(1+\lambda,2))}\\
        &\lesssim (N^{-1}\max|\psi'|)^k\tnorm{\chi}
              +N^{-k}\norm{\chi}{L^2((\lambda,1)\cup(1+\lambda,2))},
    \end{align*}
    where an inverse inequality 
    was used. Combining the results finishes the proof.
  \end{proof}
  \begin{remark}
    We could have also used a different layer adapted mesh, like a Dur\'{a}n mesh, introduced in \cite{Duran}, modified to our problem.
    The proof of interpolation errors and finally convergence follows again the standard ideas.
  \end{remark}  

  \subsection{On a coarser mesh}
  Let us consider a mesh, see also \cite{Roos22} where a similar mesh is used for weak layers, 
  that resolves the weak layer not by an S-type mesh, but just by an even simpler equidistant mesh and a specially
  chosen transition point, while the strong layer is still resolved by an S-type. Thus let
  \[
    \lambda:=\frac{\sigma\eps}{\beta}\ln N\leq\frac{1}{2}
    \qmbox{and}
    \mu:=\frac{\eps^\frac{k-1}{k}}{\beta}\leq\frac{1}{2}
  \]
  that still implies the weak condition
  \[
    \eps\lesssim (\ln N)^{-1}.
  \]
  Note that in the case $k=1$ we set $\displaystyle \mu=\frac{1}{2}$.
  The by the same ideas as in the previous subsection 
  we construct the mesh nodes
  \[
    x_i=\begin{cases}
           \displaystyle \frac{\sigma\eps}{\beta}\phi\left(\frac{2i}{N}\right),&           \displaystyle 0\leq i \leq \frac{N}{4},\\[1ex]
           \displaystyle \frac{4i}{N}(1-\lambda)+2\lambda-1,                   & \displaystyle \frac{N}{4}\leq i \leq \frac{N}{2},\\[1ex]
          \displaystyle  1+\mu\left(\frac{4i}{N}-2\right),                     & \displaystyle \frac{N}{2}\leq i \leq \frac{3N}{4},\\[1ex]
           \frac{4i}{N}(1-\mu)+4\mu-2,                           & \displaystyle \frac{3N}{4}\leq i \leq N.
        \end{cases}
  \]
  Note that for $i\geq N/4$ it is always piecewise equidistant, independent of the choice of $\phi$.
  For the (minimal) mesh width in the different regions it holds
  \[
    h_1\lesssim \eps,\,
    H_1\sim N^{-1},\,
    h_2\sim N^{-1}\eps^{\frac{k-1}{k}}
    \qmbox{and}
    H_2\sim N^{-1}.
  \]
  The proof of the interpolation errors uses local interpolation error estimates, given on any cell 
  $\tau_i$ with width $h_i$ and $1\leq s \leq k+1$ and $1\leq t\leq k$ by
  \begin{subequations}
    \begin{align}
      \norm{v-Iv}{L^2(\tau_i)}   &\lesssim h_i^{s}\norm{v^{(s)}}{L^2(\tau_i)},\label{eq:inter:1}\\
      \norm{(v-Iv)'}{L^2(\tau_i)}&\lesssim h_i^{t}\norm{v^{(t+1)}}{L^2(\tau_i)},\label{eq:inter:2}
    \end{align}
  \end{subequations}
  for $v$ smooth enough. In principle it is similar to proving interpolation error estimates on S-type meshes
  but the different layout of the mesh makes some changes in the proof necessary.
  
  \begin{lemma}
    Let us assume $\sigma\geq k+1$ and 
    \begin{align}\label{ass:eps}
      \e^{-\eps^{-1/k}}\leq N^{1-k}.
    \end{align}
    Then it holds
    \begin{subequations}
      \begin{align}
        \norm{u-Iu}{L^2(\Omega)}&\lesssim (h_1+N^{-1}\max|\psi'|)^{k+1/2},\label{eq:intest:1}\\
        \tnorm{u-Iu}&\lesssim (h_1+N^{-1}\max|\psi'|)^k\label{eq:intest:2}
      \end{align}
      and more detailed 
      \begin{align}
        \norm{W-IW}{L^2(\Omega)} &\lesssim \eps^{1/2}N^{-k},\label{eq:intest:5}\\
        \norm{E-IE}{L^2((\lambda,1)\cup(1+\mu,2))} &\lesssim N^{-(k+1)},\label{eq:intest:4}\\
        \norm{E-IE}{L^2((0,\lambda)\cup(1,1+\mu))} &\lesssim \eps^{1/2}(N^{-1}\max|\psi'|)^k.\label{eq:intest:3}
      \end{align}
    \end{subequations}
  \end{lemma}
  \begin{proof}
    Using \eqref{eq:inter:1} and \eqref{eq:inter:2} with $s=k+1$ and $t=k$, resp. we obtain
    \begin{align*}
      \norm{S-IS}{L^2(\Omega)}&\lesssim (h_1+H_1+h_2+H_2)^{k+1}\lesssim (h_1+N^{-1})^{k+1},\\
      \norm{(S-IS)'}{L^2(\Omega)}&\lesssim (h_1+H_1+h_2+H_2)^{k}\lesssim (h_1+N^{-1})^{k}.
    \end{align*}
    For $E$ we can proceed as on a classical S-type mesh and obtain with \eqref{eq:inter:1} and $s=k+1$
    \begin{equation}
      \norm{E-IE}{L^2(0,\lambda)}\lesssim \eps^{1/2}(N^{-1}\max|\psi'|)^{k+1},\label{eq:Eest}
    \end{equation}
    while with a triangle inequality and the $L^\infty$-stability of $I$ it follows
    \begin{equation}
      \norm{E-IE}{L^2(\lambda,2)}
        \lesssim \norm{E}{L^2(\lambda,2)}+\norm{E}{L^\infty(\lambda,2)}
        \lesssim N^{-(k+1)}
        .\label{eq:tmp2}
    \end{equation}
    With \eqref{eq:inter:2} and $t=k$ we obtain 
    \[
      \norm{(E-IE)'}{L^2(0,\lambda)}\lesssim \eps^{-1/2}(N^{-1}\max|\psi'|)^{k},
    \]
    and with a triangle and an inverse inequality
    \[
      \norm{(E-IE)'}{L^2((\lambda,1)\cup(1+\mu,2))}
        \lesssim \norm{E'}{L^2(\lambda,2)}+N\norm{E}{L^\infty(\lambda,2)}
        \lesssim \eps^{-1/2}N^{-k}.
    \]
    In the remaining part \eqref{eq:inter:2} with $t=k$ yields
    \begin{align*}
      \norm{(E-IE)'}{L^2(1,1+\mu)}
        &\lesssim h_2^k\norm{E^{(k+1)}}{L^2(1,1+\mu)}
        \lesssim N^{-k}\eps^{k-1}\eps^{-(k+1)}\eps^{1/2}E(1)
        \lesssim N^{-k}\eps^{-3/2}\e^{-\beta/\eps}\\
        &\lesssim \eps^{-1/2}N^{-k}
    \end{align*}
    due to
    \begin{equation}\label{eq:epsbound}
      \eps^{-1}\e^{-\beta/\eps}\leq \frac{1}{\e\beta}.
    \end{equation}
    For the estimation of $W$ we follow the idea given in \cite{RR11} and apply \eqref{eq:inter:1} with $s=1$ and $s=2$ in order to obtain
    \begin{align}
      \norm{W-IW}{L^2(1+\mu,2)}
        &\lesssim N^{-1}\norm{W'}{L^2(1+\mu,2)}
        \lesssim N^{-1}\eps^{-1/2}W(1+\mu)
        \lesssim N^{-1}\eps^{1/2}\e^{-\eps^{-1/k}},\label{eq:tmp}\\
      \norm{W-IW}{L^2(1+\mu,2)}
        &\lesssim N^{-2}\norm{W''}{L^2(1+\mu,2)}
        \lesssim N^{-2}\eps^{-1/2}\e^{-\eps^{-1/k}}.\notag
    \end{align}
    Combining these results we have
    \[
      \norm{W-IW}{L^2(1+\mu,2)}
        \lesssim N^{-3/2}\e^{-\eps^{-1/k}}
        \lesssim N^{-(k+1/2)},
    \]
    due to \eqref{ass:eps}.
    Note that for $k=1$ this approach can also be done on the interval $(1,2)$, see \cite{RR11}. For $k>1$ we also have
    with \eqref{eq:inter:1} and $s=k+1$
    \begin{align*}
      \norm{W-IW}{L^2(1,1+\mu)}
        &\lesssim h_2^{k+1}\norm{W^{(k+1)}}{L^2(1,1+\mu)}
        \lesssim N^{-(k+1)}\eps^{\frac{1}{2}-\frac{1}{k}}.
    \end{align*}
    For the derivative we obtain using \eqref{eq:inter:2} with $t=k$ and $t=1$, resp.
    \begin{align*}
      \norm{(W-IW)'}{L^2(1,1+\mu)}
        &\lesssim h_2^{k}\norm{W^{(k+1)}}{L^2(1,1+\mu)}
        \lesssim \eps^{-1/2}N^{-k},\\
      \norm{(W-IW)'}{L^2(1+\mu,2)}
        &\lesssim N^{-1}\norm{W''}{L^2(1+\mu,2)}
        \lesssim \eps^{-1/2}N^{-1}\e^{-\eps^{-1/k}}
        \lesssim \eps^{-1/2}N^{-k},
    \end{align*}
    due to \eqref{ass:eps}.
    Collecting the individual results gives \eqref{eq:intest:1} and \eqref{eq:intest:2}. 
    
    With \eqref{eq:inter:1} and $s=k$ we also obtain
    \begin{align*}
      \norm{W-IW}{L^2(1,1+\mu)}
        &\lesssim h_2^k\norm{W^{(k)}}{L^2(1,1+\mu)}
        \lesssim \eps^{1/2}N^{-k}
    \end{align*}
    and together with \eqref{eq:tmp} and \eqref{ass:eps} we have \eqref{eq:intest:5}.   
    
    The result \eqref{eq:intest:4} follows directly from \eqref{eq:tmp2}.
    For the final results on $E$ we apply \eqref{eq:inter:1} with $s=k$
    and obtain
    \begin{align*}
      \norm{E-IE}{L^2(1,1+\mu)}
        &\lesssim h_2^k\norm{E^{(k)}}{L^2(1,1+\mu)}
         \lesssim \eps^{-1/2}N^{-k}\e^{-\beta/\eps}
         \lesssim \eps^{1/2}N^{-k},
    \end{align*}
    due to \eqref{eq:epsbound}. Together with \eqref{eq:Eest} we finish the proof.    
  \end{proof}

  \begin{remark}
    Assumption \eqref{ass:eps} restricts the application of the method for $k>1$ slightly. 
    We can rewrite it as
    \[
      N\leq \e^{\frac{1}{(k-1)\eps^{1/k}}}
    \]
    and Table \ref{tab:ass:eps} shows the bounds on $N$ obtained by this requirement.
    \begin{table}[htb]
      \caption{Bounds on $N$ for given $\eps$ and $k>1$\label{tab:ass:eps}}
      \begin{center}
        \begin{tabular}{l|llll}
          \diagbox{$\eps$}{$k$} & 2 & 3 & 4 & 5\\
          \hline
          1e-2   &   2.2e+04  & 10      &  2      & 1 \\
          1e-3   &   5.4e+13  & 148     &  6      & 2 \\
          1e-4   &   2.7e+43  & 47675   &  28     & 4 \\
          1e-5   &   2.7e+137 & 1.2e+10 &  375    & 12\\
          1e-6   &   2.0e+434 & 5.2e+21 &  37832  & 52
        \end{tabular}
      \end{center}
    \end{table}
    For small $k$ and reasonably small $\eps$ the coarser mesh approach can be used. For higher
    polynomial degrees, the weak layer should be resolved by a classical layer-adapted mesh like 
    the S-type mesh. Here we could still increase the value of the transition point, because
    \[
      \mu = \frac{\sigma\eps^\frac{k-1}{k}}{\beta}\ln(N)
          > \frac{\sigma\eps}{\beta}\ln(N)
    \]
    would still be enough.
  \end{remark}
  \begin{theorem}
    For the solution $u$ of \eqref{eq:Lu} and the numerical solution $u_N$ of \eqref{eq:method} 
    holds on the coarser S-type mesh with $\sigma\geq k+1$ and $\e^{-\eps^{-1/k}}\leq N^{1-k}$
    \[
      \tnorm{u-u_N}\lesssim  (h+N^{-1}\max|\psi'|)^k.
    \]
  \end{theorem}
  \begin{proof}
    The proof follows that of Theorem~\ref{thm:conv} by considering $E$ and $W$ 
    in the convective term separately and using the estimates of the previous lemma.
  \end{proof}

  \section{Numerical example}\label{sec:NumEx}
    Let us consider as example the following problem
    \begin{align*}
      -\eps u''(x)-(2+x)u'(x)+(3+x)u(x)-d(x)u(x-1)&=3,\,x\in(0,2),\\
      u(2)&=0,\\
      u(x)&=x^2,\,x\in(-1,0],
    \end{align*}
    where 
    \[
     d(x) = \begin{cases}
              1-x,& x<1,\\
              2+\sin(4\pi x),&x\geq 1.
            \end{cases}
    \]
    Here the exact solution is not known.
    On a Bakhvalov-S-mesh with $\sigma=k+1$ and $\eps=10^{-6}$ we obtain the results listed in Table~\ref{tab:S}.
    \begin{table}[htb]
      \caption{Errors $\tnorm{u-u_N}$ on a Bakhvalov-S-mesh\label{tab:S}}
      \begin{center}
        \begin{tabular}{rlclclclc}
          $N$ & \multicolumn{2}{c}{$k=1$} & \multicolumn{2}{c}{$k=2$} & \multicolumn{2}{c}{$k=3$} & \multicolumn{2}{c}{$k=4$}\\
          \hline
           16 & 1.27e-01 & 0.96 & 2.17e-02 & 1.94 & 3.53e-03 & 2.89 & 5.74e-04 & 3.85\\
           32 & 6.52e-02 & 0.98 & 5.68e-03 & 1.96 & 4.77e-04 & 2.94 & 3.98e-05 & 3.92\\
           64 & 3.31e-02 & 0.99 & 1.46e-03 & 1.98 & 6.20e-05 & 2.97 & 2.63e-06 & 3.96\\
          128 & 1.67e-02 & 0.99 & 3.69e-04 & 1.99 & 7.92e-06 & 2.98 & 1.69e-07 & 3.98\\
          256 & 8.36e-03 & 1.00 & 9.29e-05 & 1.99 & 1.00e-06 & 2.99 & 1.08e-08 & 3.98\\
          512 & 4.19e-03 & 1.00 & 2.33e-05 & 1.99 & 1.26e-07 & 3.00 & 6.81e-10 & 3.71\\
         1024 & 2.10e-03 &      & 5.87e-06 &      & 1.58e-08 &      & 5.21e-11 &     \\
        \end{tabular}  
      \end{center}
    \end{table}
    For other values of $\eps$ the results are similar. Obviously we see the expected rates of $N^{-k}$
    in $\tnorm{u-u_N}$. For the computation of these results instead of an exact solution,  
    a reference solution on a finer mesh and higher polynomial degree was used.
    
    On the coarsened mesh we obtain the results shown in Table~\ref{tab:coarse}.
    \begin{table}
      \caption{Errors $\tnorm{u-u_N}$ on the coarsened mesh\label{tab:coarse}}
      \begin{center}
        \begin{tabular}{rlclclclc}
          $N$ & \multicolumn{2}{c}{$k=1$} & \multicolumn{2}{c}{$k=2$} & \multicolumn{2}{c}{$k=3$} & \multicolumn{2}{c}{$k=4$}\\
          \hline
              & \multicolumn{8}{c}{$\eps=10^{-6}$}\\
          \hline
           16 & 1.27e-01 & 0.96 & 2.16e-02 & 1.93 & 3.53e-03 & 2.89 & 5.74e-04 & 3.85\\
           32 & 6.52e-02 & 0.98 & 5.67e-03 & 1.96 & 4.77e-04 & 2.94 & 3.98e-05 & 3.92\\
           64 & 3.31e-02 & 0.99 & 1.46e-03 & 1.98 & 6.20e-05 & 2.97 & 2.63e-06 & 3.96\\
          128 & 1.67e-02 & 0.99 & 3.69e-04 & 1.99 & 7.92e-06 & 2.98 & 1.69e-07 & 3.98\\
          256 & 8.36e-03 & 1.00 & 9.29e-05 & 1.99 & 1.00e-06 & 2.99 & 1.08e-08 & 3.98\\
          512 & 4.19e-03 & 1.00 & 2.33e-05 & 1.99 & 1.26e-07 & 3.00 & 6.81e-10 & 3.71\\
         1024 & 2.10e-03 &      & 5.87e-06 &      & 1.58e-08 &      & 5.19e-11 &     \\
        \hline
              & \multicolumn{8}{c}{$\eps=10^{-3}$}\\
        \hline
           16 & 1.27e-01 & 0.96 & 2.18e-02 & 1.93 & 3.55e-03 & 2.86 & 5.94e-04 & 3.86\\
           32 & 6.53e-02 & 0.98 & 5.72e-03 & 1.95 & 4.88e-04 & 2.96 & 4.08e-05 & 3.57\\
           64 & 3.31e-02 & 0.99 & 1.48e-03 & 1.96 & 6.26e-05 & 2.97 & 3.43e-06 & 0.66\\
          128 & 1.67e-02 & 0.99 & 3.80e-04 & 1.98 & 7.97e-06 & 2.99 & 2.17e-06 & 0.44\\
          256 & 8.38e-03 & 0.99 & 9.63e-05 & 2.03 & 1.01e-06 & 2.98 & 1.59e-06 & 0.78\\
          512 & 4.21e-03 & 1.00 & 2.36e-05 & 2.01 & 1.28e-07 & 2.77 & 9.29e-07 & 1.61\\
         1024 & 2.11e-03 &      & 5.86e-06 &      & 1.88e-08 &      & 3.05e-07 &     \\
        \end{tabular}   
      \end{center}
    \end{table}
    We observe for $\eps=10^{-6}$ almost the same numbers as for the Bakhvalov-S-mesh. 
    Here the conditions of Table~\ref{tab:ass:eps} are fulfilled and we do not observe a 
    reduction in the orders of convergence. But for the larger $\eps=10^{-3}$ there is a 
    visible reduction in the convergence orders for $k=4$. This demonstrates clearly, that
    for higher polynomial degrees and rather large $\eps$ a classical layer adapted mesh 
    should be chosen.
    
    \vspace{1em}
  {{\bf Acknowledgment}. The first author is supported by the Ministry of Education, Science and Technological Development of the Republic of Serbia under
    grant no. 451-03-68/2022-14/200134, while the first, second and third authors are supported by the bilateral project "Singularly perturbed problems with multiple parameters" between Germany and Serbia, 2021-2023 (DAAD project 57560935).}

  \bibliographystyle{plain}
  \bibliography{lit}

\appendix
  \section{Expansion of terms involving Green's function}\label{app:a}
    When estimating $\alpha$ in \eqref{eq:alpha} we have $\displaystyle \alpha=\frac{N}{D},$ 
 where
    \begin{align*}
            N&:=\beta-\delta+\int_0^1 G_x(0,t)\left(f(t+1)-\delta\hat c(t)+(\beta-\delta)\hat b(t)-d\int_0^1G(t,s)f(s)\ds\right)\dt\\
        &\hspace{2cm}-\int_0^1 G_x(1 ,t)f(t)\dt,\notag\\
      D&:=2+\int_0^1 G_x(1,t)\hat b(t)\dt+\int_0^1 G_x(0,t)\left(\hat b(t)+\hat c(t)+d\int_0^1G(t,s)\hat b(s)\ds\right)\dt,\notag
    \end{align*}
    and $\hat b(t):=b-ct$ and $\hat c(t):=c+dt$. The Green's function $G$ is defined as
    \[
      G(x,t) = \begin{cases}
                \displaystyle -\frac{v_1(x)v_2(t)}{\eps w(t)},&x\leq t,\\[1ex]
                \displaystyle -\frac{v_1(t)v_2(x)}{\eps w(t)},&x>t,
               \end{cases}
    \]
    where $v_1$ is the solution of
    \[
      -\eps v_1''(x)-b v_1'(x)+cv_1(x)=0,\,
      v_1(0)=0,\,v_1'(0)=1,
    \]
    $v_2$ is the solution of
    \[
      -\eps v_2''(x)-b v_2'(x)+cv_2(x)=0,\,
      v_2(1)=0,\,v_2'(1)=1
    \]
    and $w$ their Wronskian 
    \[
      w(x)=v_1(x) v_2'(x)-v_1'(x)v_2(x).
    \]
    We can expand the terms in $D$ in powers of $\eps$, here done using 
    the symbolic math program MAPLE, and obtain
    \begin{align*}
      \int_0^1 G_x(1,t)\dt  &= -\frac{1}{b}+\ord{\eps},\\
      \int_0^1 G_x(1,t)t\dt &= -\frac{1}{b}+\ord{\eps},\\
      \int_0^1 G_x(0,t)\dt  &=  \frac{b}{c} \left(1-\e^{-\frac{c}{b}}\right)\frac{1}{\eps} 
                               + \frac{b - (2b + c)\e^{-\frac{c}{b}}}{b^2}
                               +\ord{\eps},\\
      \int_0^1 G_x(0,t)t\dt &=  \frac{b}{c^2}(b - (b + c)\e^{-\frac{c}{b}})\frac{1}{\eps} 
                              + \frac{1}{b^2c}(2b^2 - (2b^2 + 3bc + c^2)\e^{-\frac{c}{b}})
                              +\ord{\eps},
    \end{align*}
    \begin{align*}
      \int_0^1 G_x(0,t)\int_0^1G(t,s)\ds\dt &=   \frac{b - (b + c)\e^{-\frac{c}{b}}}{c^2}\frac{1}{\eps} 
                                                - \frac{b + c}{b^3}\e^{-\frac{c}{b}}
                                                + \ord{\eps},\\
      \int_0^1 G_x(0,t)\int_0^1G(t,s)s\ds\dt &=   \frac{2b^2 - (b^2 + (b+c)^2)\e^{-\frac{c}{b}}}{c^3}\frac{1}{\eps},\\&\hspace*{2cm}
                                                + \frac{2b^3 - (2b^3 + 2b^2c + 2bc^2 + c^3)\e^{-\frac{c}{b}}}{c^2b^3} 
                                                + \ord{\eps}.
    \end{align*}
    Combining these expansions into the denominator $D$ we have
    \[
      D = (b + d\e^{-\frac{c}{b}})\frac{1}{\eps}
          -\frac{d}{b^3}(2b^2 - c^2)\e^{-\frac{c}{b}} 
          +\ord{\eps}
    \]
    and therefore, remember $b,d>0$, it follows $\displaystyle D\gtrsim \frac{1}{\eps}$. Using above expansions again, we also have
    \begin{align*}
      \int_0^1|G_x(1,t)|\dt &\lesssim 1,              &\int_0^1|G_x(0,t)|\int_0^1 G(t,s)ds\dt  & \lesssim \frac{1}{\eps},\\
      \int_0^1|G_x(0,t)|\dt &\lesssim \frac{1}{\eps}, &\int_0^1|G_x(0,t)|\int_0^1 G(t,s)sds\dt & \lesssim \frac{1}{\eps},\\
      \int_0^1|G_x(0,t)|t\dt&\lesssim \frac{1}{\eps},
    \end{align*}
    and can estimate the numerator $N$
    \[
      |N| \lesssim |\beta|+|\delta|+\left(|\beta|+|\delta|+\norm{f}{L^\infty(0,2)}\right)\cdot\frac{1}{\eps}.
    \]

\end{document}